\documentclass[twocolumn,journal]{IEEEtran}
\usepackage{graphicx}
\usepackage{amsmath,amssymb}
 \newtheorem{theorem}{Theorem}[section]
 
 \newtheorem{lemma}[theorem]{Lemma}

 \newtheorem{remark}[theorem]{Remark}
\begin{document}

\title{\LARGE From Packet to Power Switching: Digital Direct Load Scheduling}

\author{Mahnoosh Alizadeh\IEEEauthorrefmark{1},~\IEEEmembership{Student Member,~IEEE},~Anna Scaglione\IEEEauthorrefmark{1},~\IEEEmembership{Fellow,~IEEE}
 	\\and~Robert J. Thomas\IEEEauthorrefmark{2},~\IEEEmembership{Life Fellow,~IEEE}
\thanks{\IEEEauthorrefmark{1}Department of Electrical and Computer Engineering,
      University of California Davis, email: \mbox{$\lbrace$malizadeh,ascaglione$\rbrace$@ucdavis.edu} ,
\IEEEauthorrefmark{2}Department of Electrical Engineering, Cornell University, email: \mbox{rjt1@cornell.edu}.
This work has been funded by DOE under CERTS. Parts of this work was presented at Smartgridcomm 2011.
}}

%%% ----------------------------------------------------------------------

%%% ----------------------------------------------------------------------
\markboth{Accepted by the IEEE journal of Selected Areas in Communications (JSAC): Smart Grid Communications series, to appear.}{Accepted to IEEE journal of Selected Areas in Communications (JSAC): Smart Grid Communications series, to appear.}  \maketitle
%%% ----------------------------------------------------------------------
\begin{abstract}
At present, the power grid has tight control over its dispatchable generation
capacity but a very coarse control on the demand. Energy
consumers are shielded from making price-aware decisions, which degrades the efficiency of the market. This state of affairs tends to favor fossil fuel generation
over renewable sources. Because of the technological difficulties of storing electric energy, the quest for
mechanisms that would make the demand for electricity controllable on a day-to-day basis is gaining prominence. The goal of this paper is to provide one such
mechanisms, which we call Digital Direct Load Scheduling (DDLS). DDLS is a direct load control mechanism in which we unbundle individual requests for energy and digitize them so that they can be
automatically scheduled in a cellular architecture. Specifically, rather than storing energy or interrupting the job of appliances, we choose to hold
requests for energy in queues and optimize the service time of individual appliances
belonging to a broad class which we refer to as ``deferrable loads''. The function
of each neighborhood scheduler is to optimize the time at which these appliances start to function. This
process is intended to shape the aggregate load profile of the neighborhood so as to optimize an objective
function which incorporates the spot price of energy, and also allows distributed
energy resources to supply part of the generation dynamically.
\end{abstract}
\section{Introduction}\label{intro}
%An advantage of the electric power grid, in comparison to other energy delivery systems, is that electricity can be produced in many ways and distributed rapidly. Unfortunately, these same characteristics can be a problem, since not all sources of electrical energy are equally available to produce what is required to satisfy consumer demand when they ask for it.  Currently, in the U.S., much of the electric energy consumed is bought and sold through markets run by an Independent System Operator (ISO).  The mechanisms are auctions where offers are made by generators to sell specific quantities of energy at specified times.  Most auctions are cleared based on locational marginal prices while paying attention to reliability concerns.  These concerns involve planning for contingencies together with the ancillary services such as voltage support, frequency regulation, ramp rates, and so on, necessary to support the delivery of electric power from generation to load.  In addition to markets for generation, there are markets for transmission services.  For one view of how these markets operate see \cite{stoft}.  For information on specific market designs go to the ISO website of interest (e.g., CAISO, PJM, ERCOT).  

The most important rule of power systems is that demand and generation should be balanced at all times.  In order to fit into the secure operational framework of today's power grid, generators need to be able to forecast and control their output power for the next day and commit to generating the amount of electricity that they are scheduled to produce by the independent system operator (ISO). Failure to do so will result in system reliability issues; therefore, to ensure compliance, stiff penalties are imposed on those not able to meet their schedule. This favors the predictable and controllable power which can be generated from fossil fuels. With today's mostly price-inelastic demand, incorporating considerable amounts of renewable generation, given the volatility and unpredictability of such sources, is a challenge  \cite{eric}.

 One way to overcome the dispatch issues associated with renewable resources is through a cyber system in support of making energy demand both observable and controllable. One embodiment of the so-called Smart Grid is a model for \textit{smart} devices to become flexible agents, distributed across the grid, whose demand can be modified, either in a centralized or a distributed manner, via programs that produce market efficiency. These flexible agents could, among other things, modulate the demand to more easily integrate volatile renewable generation sources in the system, especially those available locally, at the edge of the distribution network. 
 %A number of techniques have been introduced in the literature for short-term load forecasting \cite{loadforecast}.
 %With the uncontrolled addition of appliances like Plug-in (Hybrid) Electric Vehicles (PHEV), smart cooling/heating systems, and the introduction of active loads like discharging EVs or Community Energy Storage (CES), volatility on the demand side of the power distribution network could be significantly increased \cite{lopes}. Consequently, traditional prediction techniques like  \cite{loadforecast} may considerably degrade in performance. 
To help this remarkable transition, utility companies are introducing smart meters, whose deployment is referred to as the Advanced Metering Infrastructure (AMI) \cite{AMI}). The AMIs are currently used to improve the prediction accuracy of the load and curtail demand during emergencies. The question is {\it when, how and with what incentives} load control could be performed on a regular basis. 
%However, even with full observability of the load which could be supplied by the meters, the current tradition of having no control over the load's behavior is not sustainable. 

%Next we summarize the emerging load control approaches. 
\vspace{-.4cm}
\subsection{Previous Work}%\vspace{-.1cm}
Previous work on demand management techniques can be categorized into either price-based load control techniques, which fall under the class of Demand Response (DR) methods, or direct load control (typically through curtailment), which is classified as Demand Side Management (DSM).  

Price-based load control strategies include Time Of Use (TOU), Critical Peak Pricing (CPP), or Real-Time Pricing (RTP) techniques, all of which require customers to make energy usage decisions individually based on pricing signals. 
In TOU pricing strategies, the price data is usually delivered months or years before the actual time of use. There have been several studies on determining these rates, which requires a dynamic model for the price response of customers, typically derived based on experiments \cite{TOU1,TOU2}. TOU can ameliorate customer behavior and smoothen the daily demand profile, but has little chance of aiding the integration of volatile resources. In RTP, instead, price information is provided only hours before consumption. With RTP, there is a need for an automated system to help the customer make energy usage decisions; generally they are referred to as Home Energy Management Systems (HEMS) (see e.g., \cite{amirhamed,Han}). HEMS plan the customers consumption given pricing data, appliance power profiles, job deadlines, and user preferences.  RTP strategies are appealing since they correct inefficiencies with a decentralized market driven control, but raise some concerns in terms of safety. Real-time prices are not actually real-time since they have to be delivered to the customer beforehand, to allow some planning time. Also, the presence of a feedback loop due to customer response may result in rebound peaks that can worsen the situation \cite{kishore} and can also lead to physical and price instability \cite{mitter}. An important question is how to compute the optimum price for each end-use customer without introducing economic and physical safety issues. Can these prices be equal for every customer, or type of appliance, in a neighborhood, or should they change? Should they be updated continuously, or in an event-driven fashion? Determining pricing also requires extensive knowledge about customer behavior: on pre-determined rates, customers are much more predictable.

Demand Side Management strategies are, in contrast, built for safety. They are usually applied directly by a control center, and require customer subscription to an economic incentive program. The first techniques emerged in the 60's and are currently employed through the so-called \textit{Interruptible Load} programs, where some of the customers appliances, upon receiving a notice, are  automatically turned off for a pre-determined amount of time (15-30 minutes).  Several methods have been proposed to determine the best load curtailment strategy \cite{LC1,LC2}.  These schemes have the benefit of aggregating several load assets but  their pitfall is that it is hard to determine off-line the degree to which they inconvenience the customers. Thus, these programs are normally used only in emergency situations. 

In recent years, the deployment of advanced two-way communication links are leading to rethink \textit{centralized load control}. Rather than the utility, an aggregator could directly control end-use appliances and tap not only customers HVAC (heating, ventilation, and air conditioning) systems but also on their electric vehicle (EV) batteries \cite{maxzhang}. Next, we explain that EVs and HVAC have significant differences and, in our view, require separate treatment.  

\subsection{Classification of Loads and research on Deferrable Loads}
There are three main classes of appliances used by residential customers. A first class of automata is one whose load profile, once on, is predictable. This class includes, for example, EVs, dishwashers, washers/dryers etc. Given that their cycles last for a long time, their starting time can be easily shifted, but their operation should not be interrupted. In this paper we specifically target this first class of {\it deferrable loads}. A second class of loads has a predictable profile except for an unknown duration. This class includes lighting, television sets, stereos etc. The service these loads deliver is extremely time sensitive, and so is their electricity demand. Therefore, the demand due to this type is not very flexible (except for dimmable lights).  Last but not least, the third class of loads includes, for example, thermostatically controlled appliances. These loads can be both interrupted and turned on earlier, but their load profile and the time sensitivity of the service depend non-linearly on a local state. Most centralized load control programs interrupt this class of appliances. Here we argue that there is an additional layer of complexity in capturing digitally their service model, and therefore, to avoid incorrect generalization, we currently exclude them from our treatment.
  
Specifically, we will name as {\it D-loads} those deferrable loads which, if enabled by embedded logic and two way communications, can specify their energy request to a remote control center and wait to be dispatched.

The most important {\it D-loads} are Electric Vehicles (EV). Researchers have proposed algorithms that control the energy use of EVs under dynamic pricing \cite{poor,korea}. Examples of direct control are the work in \cite{caramanis}, where the charging rate of EVs drawing from a single feeder is oprimized by dividing them into service queues depending on their desired time of departure; the work in \cite{somayeh}, instead, optimizes the charging rate of EVs by including them in the optimal power flow problem solved by the system operator.  The inclusion of deferrable loads in the energy market in the presence of renewables is the subject of \cite{homer,oren1}. In this line of work, the demand response assets are modeled as an aggregate tank/storage capacity that needs to be filled by a certain deadline. The DSM model we propose in this paper can provide a more accurate description of the load modification capability of an aggregator, and for the information flows to support the service. 

\vspace{-0.1in}
\subsection{Contribution}
In this paper, we propose a direct load management scheme that mediates between the central control model of DLC and the {\it laissez faire} nature of RTP. 
The scheme we propose in this paper is a direct control architecture that can be implemented as a voluntary program in which customers release the control of the time at which their D-loads start functioning to their associated neighborhood controller or {\it aggregator}.  They turn over this control in return for a financial reward for the inconvenience they experience, as their energy use is manipulated to follow a desired demand profile closely. We envision that the scheduling decisions are made at several distributed control centers that aggregate energy requests from many customers and make energy usage decisions that are relayed back to them.  Anticipating that aggregating energy requests up to a certain level will give us the benefits of scale, our model allows us to explore to what degree cooperation among customers benefits the efficiency of the system. 
In fact, our scheme can be viewed as a model for building up {\it reservoirs of D-loads requests}, like a virtual water dam for loads, which can be dispatched optimally to follow the generation, thereby reversing the traditional roles played by generators and loads.

Our contribution is the introduction of a Digital Direct Load Scheduling (DDLS) model to manage D-loads and the network architecture that incorporates it. The model includes 1) an {\it encoding component} for controllable load, generation and for the feedback signals, 2) a {\it load scheduler} and 3) a {\it generation market interface}.  The encoding component allows deferrable appliances to become D-loads, which means that through the DDLS these appliances become capable of specifying their request for energy and release the control of their starting time to the scheduler, which determines their activation time in such a way to shape the load profile to follow available green generation. In this paper, we present the encoding component and the load scheduler aspects of our proposed model.  The proposed feedback mechanism informs the D-loads that need to be activated about the time they can start their job.  The description of how the proposed management system interacts with the electricity market and the extension of the DDLS to thermostatically controlled appliances is left as future research.

\section{The load model}\label{sec.model}\label{basic}

The load in the power grid is a mixture of random requests. In our model we separate the D-loads portion from the remaining part of the load. In fact, the load is:
\begin{equation}\label{loadeq}
L(t) = L^N(t) + L^{S}(t),
\end{equation} 
where $L^N(t)$ represents the base load, which we presumably have no control over, while $L^{S}(t)$ is the {\it controllable part of the load} due to the D-loads.  Here, we consider $L^N(t)$  predictable using standard load forecasting methods (e.g., \cite{m1,m2,m3}).

The next step is to further fragment $L^S(t)$ into contributions from individual appliances. In the model, D-loads arrive in the system randomly following a non-stationary  arrival process:
\begin{eqnarray}\label{arrival}
a(t)=\sum_{i=1}^{\infty}u(t-t_{i}^a),
\end{eqnarray}
with $u(t)$ the unit step, $t_{i}^a$ the arrival time of the $i^{th}$ D-load of any type, and an arrival rate of $\lambda(t)$ for these appliances. Each arrival event $i$ has an associated parameter vector $C_{i}$ that determines uniquely the time evolution of the load contribution when that appliance is turned on. In fact, we assume that, if turned on at time zero, the D-load injection is the complex phasor signal $g(t;C_i)$, one to one with $C_{i}$. In general, $C_{i}\in \mathbb{C}^{N}$ can be the Nyquist samples, or the Fourier or Wavelet coefficients of the known load evolution after activation. A simple example is that of EVs, for which $C_i$ is a two-dimensional vector, 
representing the charging rate and the fraction of battery capacity needed by the car upon its arrival.  
 $C_{i}$'s are assumed to be i.i.d. random vectors with a known stationary probability distribution $f_C(c)$, independent of the arrival times. Thus, the $i$th arrival event is modeled by a tuple $(t^a_{i},C_{i})$. 
If every arriving D-load is allowed in the system without any delay, the unscheduled demand profile would be:
\begin{equation}\label{lus}
L^{US}(t) =  \sum_{i=1}^{\infty} g(t - t_{i}^a;C_{i}).
\end{equation}
The degrees of freedom we will use in shaping the load is to delay the switch-on time for each D-load. Hence, the process of D-loads switching on is a {\it departure process}, in which the departure time $t_{i}^d>t_i^a$ is the time instant when the $i^{th}$ appliance is scheduled to switch on:
\begin{equation}\label{departure}
d(t)=\sum_{i=1}^{\infty}u(t-t_{i}^d),~~~t_{i}^d\geq t_{i}^a.
\end{equation}
In addition, it is convenient to include in $L^N(t)$ the load due to previously scheduled D-loads, since scheduled tasks are non-preemptable and we cannot reschedule them.  \footnote{This assumption is not as restrictive as it seems, since one could potentially impose an upper-bound to the request made, effectively breaking down the transaction in parts. We believe that this option would be much more practical than assuming jobs can be interrupted arbitrarily. }

Thus, the future values of $L^{S}(t)$ can be written as the following function of future departures
\begin{equation}\label{lev}
L^{S}(t) = \sum_{i\in I} g(t - t_i^d;C_i),~~~~t_i^d\geq t_i^a,
\end{equation}
where the set $I$ includes all D-loads that have yet not been scheduled, including those that have already arrived in the system at time $t$ but have not yet been authorized to function or the future arrivals.

\section{The Encoding Component of the DDLS}\label{commmod}\label{datarep}
%\subsection{Encoding the Demand from D-loads}
To enable the optimal scheduling, the $i$th D-load communicates the tuple $(t_i^a,C_i)$ to the scheduler, upon its arrival.
Naturally, this information has to be digitized. 

First, the charging codes $C_i$ are quantized through a mapping $\Psi(C_i)$ onto $Q$ codes $C_q,~q=1,\ldots,Q$, forming the codebook ${\mathcal C}=\{C_1,\ldots,C_Q\}$. The effect of this mapping is twofold: 1) it obviously  
provides a digital representation of the incoming request that can be communicated; 2) it separates the  D-loads in $Q$ classes of service (queues), one per quantization code. 
The arrival rate for the $q$-th queue is:
\begin{equation*}
\lambda_{q}(t) = \lambda(t) \int_{x \in \Psi^{-1}(C_{q})}f_C(x)dx,~~~q=1,\ldots,Q.
\end{equation*}

Since arrivals in different queues are independent processes, the arrival process $a(t)$ can be divided into $Q$ separate processes whose state can be represented by a vector $\bar{a}(t) = [a_1(t) \ldots a_Q(t)]^T$ of length $Q$ with the property that $\Vert \bar{a}(t)\Vert_1 = a(t)$. Each element of this vector is given by,
\begin{equation}
a_q(t)=\sum_{i=1}^{\infty} \delta(\Psi(C_i) - C_q) u(t-t_{i}^a),
\end{equation}
where $\delta(.)$ here represents the Kronecker delta function. The corresponding departure processes can also be represented by a vector $\bar{d}(t) = [d_1(t) \ldots d_Q(t)]^T$, also satisfying $\Vert \bar{d}(t)\Vert_1 = d(t)$. Each $d_q(t)$ is:
\begin{equation}
d_q(t)=\sum_{i=1}^{\infty} \delta(\Psi(C_i) - C_q) u(t-t_{i}^d).
\end{equation}
 We also know that $\bar{a}(t) \succeq \bar{d}(t)$, where $\succeq$ represents element by element inequality. This is simply due to the fact that the the number of departures from each queue can never be larger than the number of arrivals.
The quantization of the charge codes allows a simple system representation of the relationship between the individual queue departure processes and the total D-load, that is equivalent to \eqref{lev} for $Q\rightarrow \infty$. In fact, since the $C_{q}$'s are discrete values, the synthesis formula for the quantized load reconstruction is:
\begin{eqnarray}\label{conv}
\hat L^{S}(t)&= &\sum_{q=1}^{Q} \sum_{i\in I} \delta(\Psi(C_i) - C_q) g(t - t_i^d;C_q)\\
& =& \sum_{q=1}^{Q} \frac{\partial}{\partial t}d_q(t)\star g({t};{C_{q}}),
\end{eqnarray}
where  $\frac{\partial}{\partial t}{d}_q(t)$ will produce a Kronecker delta at each $t_i^d$ in the $q$-th queue and $\star$ is the convolution sign. 

In order to work online and compute feasible schedules, the second step is to operate in discrete time intervals. Thus, we assume that the departure process $\bar{d}(t)$ can only have increments at discrete times:
\begin{equation}
t_{i}^d \in \lbrace \ell\triangle |\ell \in \mathbb{N}\rbrace.
\end{equation}
 Given the fact that both $t_{i}^d$'s and $C_q$'s are chosen from a discrete set, we can alternately represent the future D-loads contribution in \eqref{lev} for $t\geq \ell_0\triangle = \lceil\frac{t}{\triangle}\rceil\triangle$ as
\begin{equation}
\hat L^{S}(t)= \sum_{q=1}^{Q} \sum_{\ell=\ell_0}^{\infty} [{d}_q(\ell\triangle)-{d}_q((\ell-1)\triangle)] g({t-\ell\triangle};{C_{q}}).
\end{equation}
One can further assume that $\triangle^{-1}$ is higher than the Nyquist rate needed to sample the charge profile  $g(t;C)$ for all possible $\Psi(.)$ and that $g(t;C)$ is roughly constant over intervals of length $\triangle$. In this case, not only can the load be uniquely reconstructed from its samples but the arrival processes can also be represented in discrete time. From this point on we use a discrete model for arrivals, departure processes and for the load profile, and replace the previous quantities with their discrete-time counterparts:
\begin{equation} \left. \begin{array}{ll}
a_q(\ell\triangle)\rightarrow a_q(\ell),~~~d_q(\ell\triangle)\rightarrow d_q(\ell),~~~ \\
g({\ell\triangle};{C_{q}}) \rightarrow g_q(\ell),~~~~
\hat L^S(\ell\triangle) \rightarrow L^S(\ell).
\end{array} \right.
\end{equation}
The discrete-time load synthesis formula  from the decisions $d_{q}(\ell)$ and samples of the pulses $g_{q}(\ell)$ is:
\begin{eqnarray}\label{disload}
 L^{S}(\ell)&=& \sum_{q=1}^{Q} \sum_{k=\ell_0}^{\infty}[d_q(k)-d_{q}(k-1)]g_{q}(\ell-k) \nonumber\\
&=&  \sum_{q=1}^{Q} J d_{q}(\ell) \star g_{q}(\ell), ~~~~~~\ell\geq \ell_0,
\end{eqnarray}
where $J$ is the first difference operator. Fig. \ref{model} relates the queues states with the synthesis formula \eqref{disload}.

\begin{figure}
\centering
\includegraphics[scale=0.1]{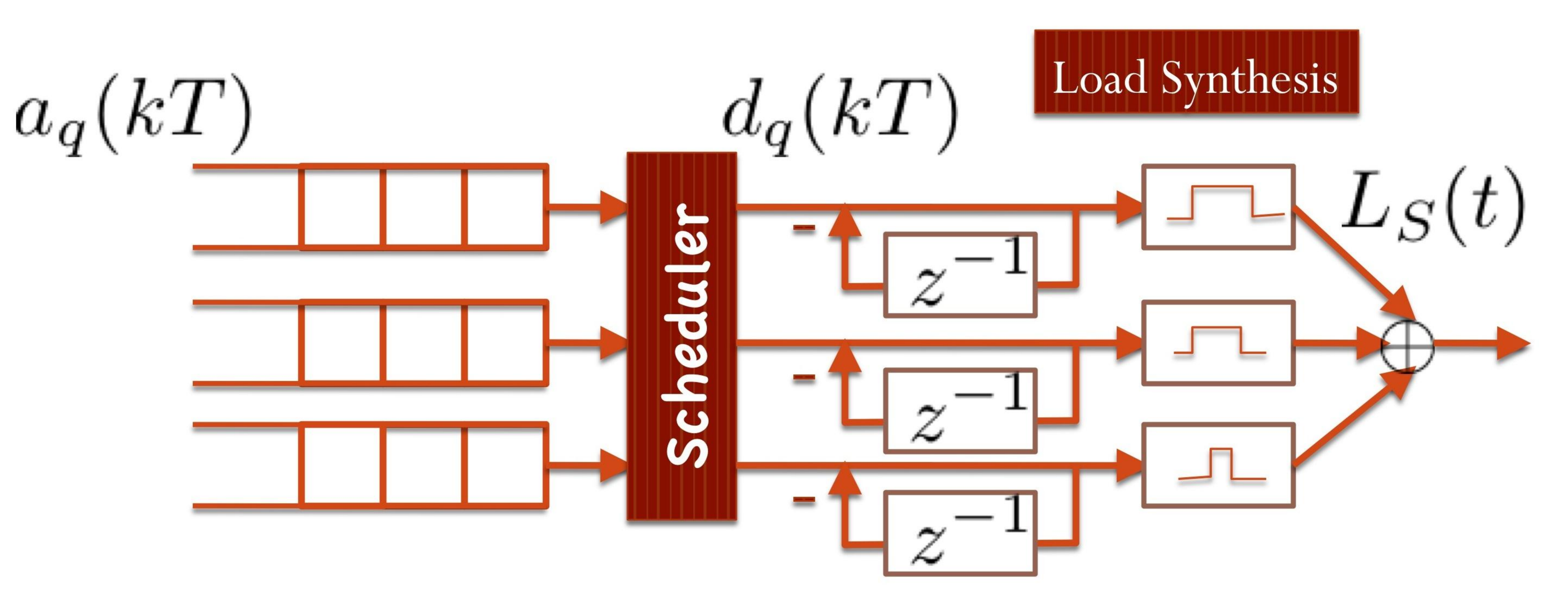}
\caption{Diagram of the Load Synthesis Equation}
\label{model}
\end{figure}

\subsection{The Inconvenience Cost: Average Delay Model}\label{dci}
In order to measure the customers rewards, another quantity that requires modeling is the delay experienced by participants in the DDLS program. Scheduling the energy use of appliances is typically accompanied by deadlines preferred by the consumer, similar to most job scheduling algorithms. But, as in many other scheduling problems that handle diverse job profiles, incorporating deadlines will make the algorithm increasingly complex and not scalable. Specifically, with a random and non-stationary task arrival and resource model, incorporating deadlines would make our problem definition not suitable to handle a large scale customer base. Hence, we follow the path of many scalable computing and communication systems existing today that provide best-effort scheduling services without providing guarantees about the timing of resource delivery to an individual transaction (a great example of which would be the Internet, which currently provides one single class of {\it best-effort} service to voice and video traffic, which are typically severely degraded by delay). Thus, our algorithm will tolerate occasional deadline misses and will provide the maximum average QoS to the entire population and in return, with much lower communication and computational requirements. This is reasonable since we imagine that DDLS will be voluntary program and only customers with enough degrees of flexibility in their demand will participate in it. Customers with stricter deadlines shall resort to conventional service providers or be charged more than those accepting best-effort services.

It is a well known result in traffic flow theory that the total delay experienced by the customers in a queue, i.e. $\sum_i (t_i^d - t_i^a)$, is equal to the area of the queue polygon, that represents the state of the queue $s(t)$ versus time, and is obtained by superimposing the departure and arrival profiles, i.e.,
\begin{equation}
s(t) = a(t) - d(t).
\end{equation}
The delay experienced in the past by the D-loads currently present in the system is not amendable and so, it is not useful in the formulation of the optimization. Hence, for the purpose of optimal scheduling in the future, one can replace the total delay cost with the total delay experienced by all the customers from now on, which we call the {\it Delay Cost Increment} (DCI):
\begin{equation}\label{delay}
\mbox{DCI}(t) \triangleq C_I \int_t^{\infty} s(\tau)d\tau = C_I \int_t^{\infty} (a(\tau) - d(\tau))d\tau,
\end{equation}
where $C_I$ is the cost per unit of time delay\footnote{For a time-dependent delay cost 
$\mbox{DCI}(t) = \int_t^{\infty} C_I (\tau) s(\tau)d\tau.$}.

%When the traffic is subdivided in queues, since $\bar{a}(t) \succeq \bar{d}(t)$, we know that $\bar{s}(\tau)\succeq 0$ and, thus:
%\begin{eqnarray}\label{delay2}
%\mbox{DCI}(t)&=& 
%C_I \int_t^{\infty} (\Vert \bar{a}(\tau)\Vert_1 - \Vert \bar{d}(\tau)\Vert_1)d\tau \\
%&=& C_I \int_t^{\infty} \Vert \bar{a}(\tau) - \bar{d}(\tau)\Vert_1d\tau =
%C_I \int_t^{\infty} \Vert \bar{s}(\tau)\Vert_1d\tau \nonumber,
%\end{eqnarray}
%where the $Q\times 1$ vector $\bar{s}(\tau)$ is the state of each of the $Q$ queues at time $\tau$.

 With quantization, it is also possible to define different delay costs $C_{I,q}$  for each service queue, allowing to offer higher QoS to customers in return for higher subscription rates.  To do so, we define a matrix $C_T = \mbox{diag}[C_{I,1}, \ldots C_{I,Q}]$:
\begin{equation}\label{delay5}
\mbox{DCI}(t)= \int_t^{\infty}\left\Vert\ C_T(\bar{a}(\tau)-\bar{d}(\tau)) \right\Vert_1 d\tau.
\end{equation}
In addition, when using a discrete set of decision epochs, 
%
%we cannot change the delay experienced by the appliances before $\ell_0\triangle = \lceil\frac{t}{\triangle}\rceil\triangle$. Applying this fact to \eqref{delay2}, we have
%\begin{eqnarray}\label{delay3}
%\mbox{DCI}(\ell_0\triangle)&=&C_I \int_{\ell_0\triangle }^{\infty} \Vert \bar{a}(\tau)-\bar{d}(\tau)\Vert_1 d\tau \\
%%&=&C_I \left\Vert \int_{\ell_0\triangle}^{\infty}[\bar{a}(\tau)-\bar{d}(\tau)]d\tau\right\Vert_1 \nonumber\\
%&=& C_I \left\Vert \sum_{\ell=\ell_0}^{\infty}\int_{\ell\triangle}^{(\ell+1)\triangle}\bar{a}(\tau)d\tau-\bar{d}(\ell\triangle)\triangle\right\Vert_1 \nonumber.
%\end{eqnarray}
%Since the delay experienced by the customers between their arrival time until the next possible decision epoch is also not amendable, we can also remove the associated penalty from our cost and thus, reformulate the delay cost as a discrete function,
\begin{equation}\label{delay4}
\mbox{DCI}(\ell_0)=\left\Vert \sum_{\ell=\ell_0}^{\infty}C_T[\bar{a}(\ell)-\bar{d}(\ell)]\right\Vert_1,
\end{equation}
Where $\ell_0\triangle = \lceil\frac{t}{\triangle}\rceil\triangle$. Note that the interchangeability of integration and the $\mathcal{L}_1$ norm is possible since $\bar{a}(\tau) - \bar{d}(\tau) \succeq 0$.

\subsection{Network model}\label{datagat}

The communication network supports two directions of communication: the uplink towards the scheduler, and the downlink activation feedback from the scheduler to the loads. 
We assume that the scheduler aggregates arrivals over a large number of customer premises, each of which pools the requests in a local Home Energy Management System (HEMS), interacting with the scheduler, which we refer as the Community Energy Management System (CEMS). While the CEMS gathers information and schedules the D-loads, another entity provides the interface to the market of generation assets.  
We call this agent, the Green Energy Management System (GEMS). 
The architecture is shown in Fig. \ref{fig.architecture}.

\begin{figure}
\centering
\includegraphics[scale=0.21]{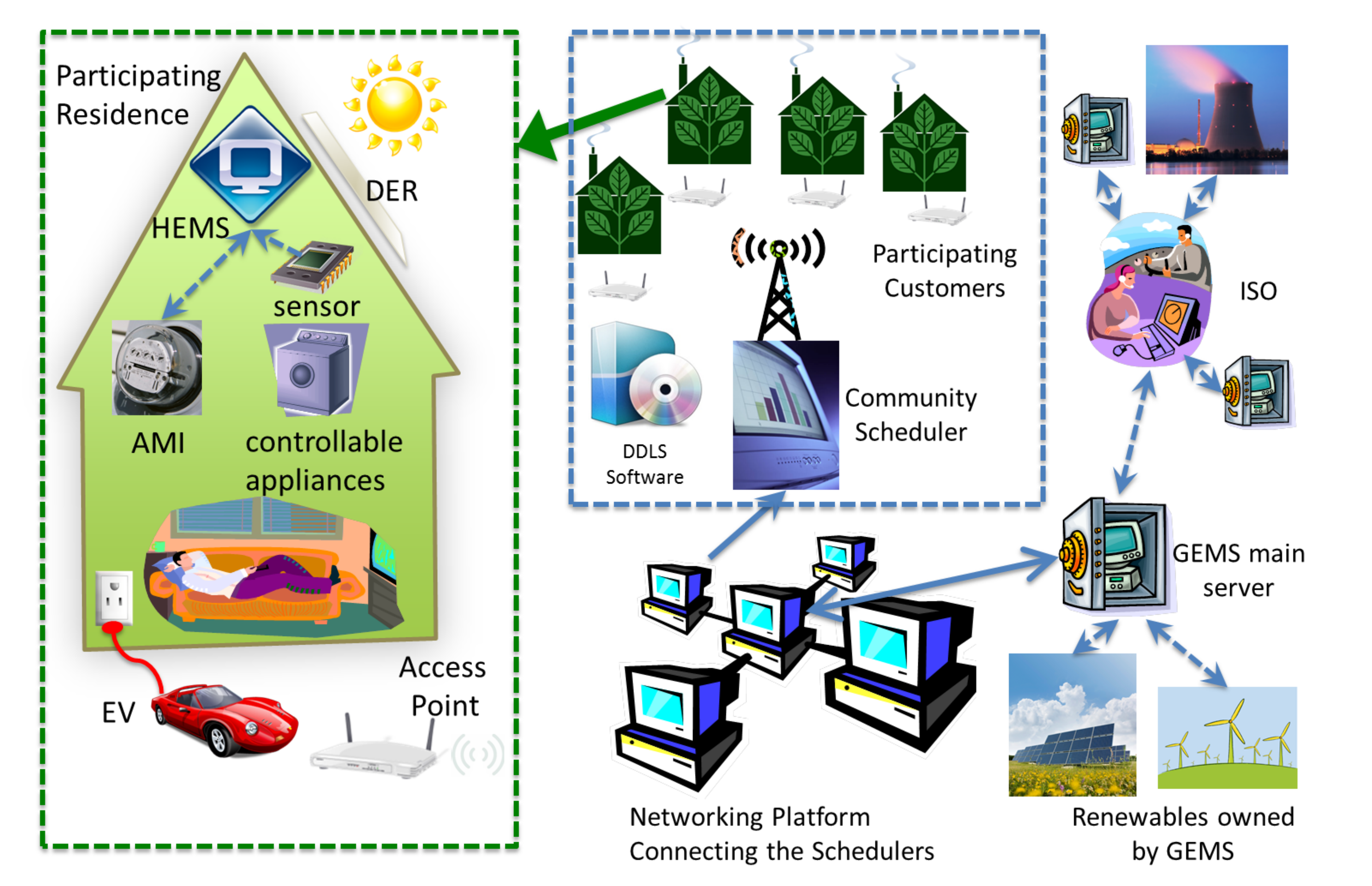}
\caption{The GEMS Architecture}
\label{fig.architecture}
\end{figure} 

The encoding component we laid out above provides a clear basis to compute exactly the rate requirements associated with the DDLS, and in this section we measure the communication rate requirements necessary to enable the control mechanism. 
\subsubsection{The Scheduler Uplink (HEMS to CEMS)}
To reconstruct and optimize the demand, the CEMS needs to know the increments of $\bar{a}(\ell \triangle)$, i.e., the arrivals in the queues.

We envision the HEMS as a software application run by non-dedicated hardware (a personal computer, or a smart-phone) and assume that the individual D-loads have a sensor network interface that allows them to convey their request to the HEMS in a local area network (via Bluetooth, Zigbee or new emerging standards for sensor network communications \cite{sensornetworks}).
HEMS locally compute the accrued inconvenience cost, and total consumption of D-loads in the premise over long time intervals. This information is communicated offline to track the quality of service delivered and for billing purposes. 
However, given that the scheduler only cares about $\bar{a}(\ell  \triangle)$, the data communicated for real time scheduling can be both anonymized, as well as aggregated as they flow towards the CEMS.
The time sensitive bits correspond to the digital communication of the tuple $(t_{i}^a,q)$. Clearly, this requires $\log Q$ bits for the charging class.
The arrival time $t_{i}^a\approx \ell_i^{a} \triangle$, where $\ell_i^{a}$ represents its finite precision representation.
Let $\ell_{i}^n>\ell_{i}^a$ be the {\it notification} time index of the arrival, i.e., the time when the event is first recorded and added to the corresponding queue.
Let $D$ be the maximum network delay in units $\triangle$,  such that with overwhelming probability,  
$\ell_{i}^n-\ell_{i}^a<D$; then, denoting by $\lfloor x\rfloor$ the floor function, this implies that $\lfloor  \ell_{i}^n/D\rfloor =\lfloor  \ell_{i}^a/D\rfloor$. Thus, the code 
$$p_i^a=\ell_{i}^a-\lfloor  \ell_{i}^a/D\rfloor$$
allows to reconstruct the arrival time as  
$\ell_{i}^a=\lfloor  \ell_{i}^n/D\rfloor+p_i^a.$
In this case, since $p_i^a\in \{0,\ldots,D-1\}$ clearly, encoding  $t_{i}^a$ only requires $\log_2 D$ bits. 
Hence, considering that $\lambda(\ell\triangle)$ is the traffic of D-loads arriving in the system, the HEMS access channel needs to support an aggregate traffic of 
\begin{equation}\label{eq.Rhems}
R_{\rm HEMS}(\ell)=\frac {1} {\triangle}\lambda(\ell\triangle)\log_{2}(DQ).
\end{equation}
 
As we discussed before, the traffic can be aggregated at the first network relay, acting as a base station (BS); for example a BS could map one to one with each area transformer or to a ISP node, by coalescing the arrival times into information about  $\bar{a}(\ell\triangle)$. If the arrivals follow a Poisson arrival process, $E\{\bar{a}(\ell\triangle)\}= \lambda(\ell\triangle)$ and the communication rate of the aggregate arrival vector is bounded by:
\begin{equation}\label{eq.Rcems}
R_{\rm CEMS}(\ell)= \frac 1 2 Q \log(2\pi e \lambda(\ell\triangle)).
\end{equation}

\subsubsection{Communicating the decisions back}
Consistent with our uplink model, we envision a downlink message structure that preserves the anonymity of the scheduled user. Once the CEMS decides the optimum schedule, it sends a feedback record, to let the vector $\bar d^{\rm opt}(\ell\triangle)$ of D-loads in. This feedback consists of a $Q \times 1$ vector $\bar{T}(\ell)$, which alerts all appliances in the corresponding classes $q=1,\ldots,Q$ that arrived before times $T_{q}(\ell)$ to enter the system.
The calculation of these vectors is performed as indicated in Fig. \ref{figdec}, i.e.,
\begin{equation}\label{alpha2}
T_{q}(\ell)=\max\lbrace \tau\leq \ell: a_{q}(\tau) \leq d_{q}^{\rm opt}(\ell)\rbrace.
\end{equation}

\begin{figure}
\centering
\includegraphics[width = 2.7in,height=1.5in]{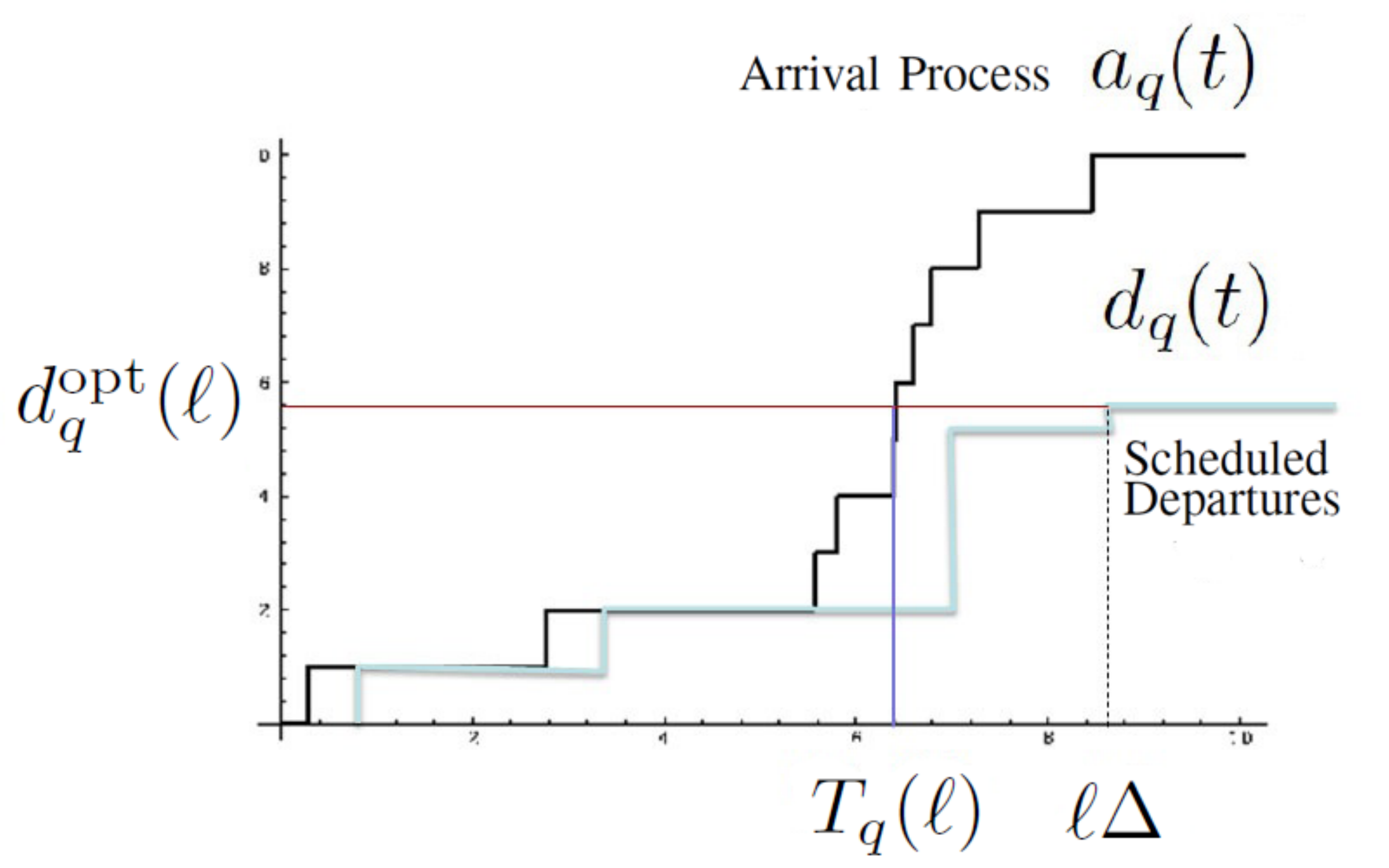}
\caption{\small Mapping decisions into feedback messages}
\label{figdec}
\end{figure}

This system makes sure that the departures match the desired value, while guaranteeing the anonymity of the information about the access. 
Also in this case it is not necessary to transmit absolute times $T_{q}(\ell)$. 
Considering that the delay is an explicit cost for the optimization, and that the optimum decisions are correlated since the queue states are correlated,  $T_{q}(\ell)$ can be differentially encoded with a relatively modest rate requirement. More specifically, assuming that $\rho_q$ is the minimum correlation coefficient among the decisions $d_q(\ell)$ for D-loads in any class and at any time, and that the variance of the delay for appliances is bounded by $\sigma_{T_q}^2$, then each feedback vector requires a number of bits per second:
$$R_{\mathrm{feedback}}={\cal O}\left(\frac 1 \triangle \sum_{q=1}^Q\frac 1 2 \log_2(e (1-\rho_q^2)\sigma_{T_q}^2)\right).$$ 

\subsection{Optimal Quantization Rate}
Suppose that $\gamma\left(g\left({t};{C}\right),g\left({t};{Q(C)}\right)\right)$ is a distortion metric, measuring the quantization error in the reproduction of the individual load profile. 
The average per queue and per load distortion $\chi_q$ caused by the load reconstruction technique is, 
\begin{align*}
		\chi_q	= \int_{t=0}^{\infty}
			\int_{x\in Q^{-1}(C_{q})}\gamma\left(g\left({t};{x}\right),g\left({t};{C_{q}})\right)\right)f_C(x)dt dx.
\end{align*}
Weighting the relative contributions of each queue by the corresponding arrival rate $\lambda_q^{\max}=\sup_t \lambda_q(t)$, the  average distortion per unit time $\chi_{tot}$ is bounded by 
\begin{equation}\nonumber
\chi_{tot} \leq \sum_{q=1}^{Q}\lambda_q^{\max}\chi_q.
\end{equation}
Hence, the number of codes/queues $Q$ can be chosen based on an appropriate rate-distortion optimization. For example:

1)~Minimize $Q$ (and, thus, the rate) while meeting a desired maximum distortion $\chi*$ in load reconstruction, i.e.,

\begin{equation*}\left. \begin{array}{ll}\label{detq}
{\min}~Q~~\mbox{s.t.}~~\sum_{q=1}^{Q}\lambda_q^{\max}\chi_q\leq \chi*.
\end{array} \right.
\end{equation*}

2)~Minimize the distortion, limiting the maximum bit rate, 
\begin{equation*}\left. \begin{array}{ll}
\min_{{\cal C}} \sum_{q=1}^{Q}\lambda_q^{\max}\chi_q\\
\mbox{s.t.} ~~\underset{\lambda}\max ~R_{\rm HEMS} \leq R^{\max}_1~\underset{\lambda}\max ~ R_{\rm CEMS} \leq R^{\max}_2
\end{array} \right.
\end{equation*}
where the rate functions $R_{\rm HEMS}$ and $R_{\rm CEMS}$ are defined in \eqref{eq.Rhems}-\eqref{eq.Rcems} respectively. It is clear that the optimum in both cases is reached when the constraints are tight.

Next we sketch a few necessary elements of the market interface of our aggregator. We leave the complete development and analysis of this market interface to future work.

\section{The Electricity Market Interface}\label{sec.elmarket}
We assume that the GEMS cannot use its local renewables to export power to the upstream distribution grid and its generated electricity can only be used locally. The overall system including  the customers, renewables, CEMS and GEMS, needs to behave as a single retailer in the wholesale power market and it can also provide ancillary services. The various operational costs that the cell may have include:
\vspace{.2cm}
\\ {1) \bf Wholesale market day-ahead bidding cost}: by looking on the day-ahead forecasts of its local generation units and its demand pattern, the market interface issues a bid in the day-ahead wholesale market to purchase a certain amount of power for every hour of the next day, so that it can safely serve all of its load reliably. 
\\ {2) \bf Wholesale market real-time bidding cost}: if local generation and day-ahead bid are insufficient to meet the demand, the market interface requests more energy from the central grid in the spot market. Also, if it has purchased extra amount of energy that it cannot consume, it can either sell it back to the grid or pay for negative spinning reserve, in case there is no demand for it. To model this cost, we assume that there is what we call the {\it supply curve} $B(t)+R(t)$, which includes the day-ahead bid plus  the the cell local renewable generation and that  there is a per unit cost of $C_{up}(t)$ for deviating upward and a per unit cost of $C_{dn}(t)$ for a downward deviation with respect to the supply curve. Subtracting the uncontrollable load $L^N(t)$ from the supply curve, we denote the Zero Incremental Cost (ZIC) power profile for D-loads by $P(t)$:
\begin{equation}\label{eq.zic}
P(t) = B(t) + R(t) - L^N(t).~~~~(\mathrm{ZIC-profile})
\end{equation}
For short look-ahead horizons, we assume that a reasonably accurate estimate of $R(t)$ is available and thus, $P(t)$ can be considered a deterministic function.
Also, we assume that near perfect estimates of the upward and downward balancing prices are available. For results on market clearing price forecasting techniques see \cite{olsson}.
\\ {3) \bf Inconvenience cost paid to customers:} This cost is paid either directly or indirectly to the customer, as an incentive to participate in DDLS programs.
We will model this cost using the concept of DCI introduced in \ref{dci}. The inconvenience cost per unit time can be variable with time of day and the queue in which the D-load is placed.

The cell should also purchase energy from its local renewable resources (if they are not owned by the cell). To cover the above mentioned costs, the cell bills its customers for the electricity it purchases from the wholesale market and on-sells to them. Also, a retailer with a DDLS program is in fact capable of providing several types of ancillary services with different dynamics and can provide peak shaving and price control services to the utility or the ISO in return for appropriate reimbursements to cover its costs. We leave a detailed discussion of these issues to our subsequent papers and instead, we will mainly focus on the scenario in which the cell has already bid in the day-ahead market $B(t)$ and has access to the information to determine the ZIC profile in \eqref{eq.zic}, and wants to minimize its costs in real-time.

\section{DDLS: The Real-time Scheduling Problem}\label{micrort}
The GEMS, which is in charge of the market interface of the cell, needs to purchase energy from the central wholesale market in order to satisfy the cell demand.  
Thus, each day, the GEMS solves an optimization problem to determine the optimal day-ahead bid $B^*(\ell), \ell = 1,\ldots,24$ to be placed in the wholesale market for each hour of the next day. For the rest of this paper, we will assume that the day-ahead bidding process is done and focus on the real-time operation of the cell scheduler. The discussion on how to place these bids based on predictions of the wholesale market price, local generation resources and taking the DDLS ability into account requires considerable amount of space and will be presented in future works (for preliminary results, see \cite{ISCCSP}).

In real-time, the goal is to minimize the accumulated cost of operation over time. This can be accomplished through a sequential decision maker. In order to derive scheduling decisions that are foresighted, we choose to use a Model Predictive Control (MPC) model. Following our previous notation,  $P(\ell)$ , $R(\ell)$, $L^{S}(\ell)$, $C_{up}(\ell)$ and $C_{dn}(\ell)$ are samples of their continuous counterparts at $t=\ell\triangle$ (the costs are also multiplied by the interval length $\triangle$).
 Assuming a look-ahead horizon of $T$ units of time, and denoting by $u_q$ the duration of the jobs of appliances in the $q$-th queue, we can formulate the general real-time scheduling problem for the cell as,
\begin{align}\label{mainopt1}
D^{*} = \mbox{argmin}
_{D}~ &\operatorname{E}_A[\mbox{Cost of retail entity in real time}]\nonumber\\ 
 \mbox{argmin}_{D}~ &\sum_{\ell=\ell_0}^{\ell_0+T}\operatorname{E}_A \lbrace C_\ell(L^N(\ell) + L^S(\ell,D),B^*(\ell),R(\ell)) \nonumber\\ &~~~~~+ \mbox{DCI}(\ell_0,D)\} 
\end{align}
$$\mbox{s.t.}~~~ \forall \ell \leq \ell_0+T, q\in \{1,\ldots,Q\}:~~ d_{q}(\ell-1) \leq d_{q}(\ell) \leq a_{q}(\ell)$$
$$~~~d_{q}(\ell) \in  \mathbb{N},~~~ d_{q}(\ell+T-u_q) = a_{q}(\ell+T-u_q),$$
where $C_\ell(.)$ represents a prediction of the \textit{real-time} cost that the cell incurs at epoch $\ell$ when purchasing electricity equal to $L^N(\ell) + L^S(D,\ell) - (B^*(\ell)+R(\ell))$ from the local intermittent resources or a central wholesale market (if negative, the retail utility may have to pay for negative spinning reserve). The matrix $D$ represents the set of scheduling decisions that will be made  in the look-ahead horizon, i.e. $D = [\bar{d}(\ell),\ldots,\bar{d}(\ell_0+T)]$ and $D^{*}$ represents its optimal value. The matrix $A=[\bar{a}(\ell),\ldots,\bar{a}(\ell_0+T)]$ represents the arrival process. $\mathcal{D}$ is the entire decision space, i.e., $D \in \mathcal{D}$ and $\mathcal{A}$ is the space defined by the arrival process, i.e., $A \in \mathcal{A}$.  $\mbox{DCI}(\ell_0,D)$ is the cumulative cost of delaying the customers, calculated in a finite horizon $T$. The first and second constraints are due to causality and the third constraint requires that all the arriving appliances finish their jobs by time $\ell_0+T$. Stricter deadlines can be imposed to provide higher QoS to certain queues.

%In this section, we will focus on defining the function $C_{\ell}(.)$ and explaining possible ways to solve the real-time scheduling problem. 
Considering the costs defined in the Section \ref{sec.elmarket}, and denoting respectively by $P_{up}(\ell) $ and $P_{dn}(\ell)$ the upward and downward deviations from the ZIC power supply $P(\ell)$ in \eqref{eq.zic} for the D-loads, the scheduling problem in \eqref{mainopt1} can be formulated as the following mixed integer optimization problem,
\begin{align}\label{mainopt2}
 &\min_{D}  ~ \operatorname{E}\lbrace \sum_{\ell=\ell_0}^{\ell_0+T} [C_{up}(\ell)P_{up}(\ell)+ C_{dn}(\ell)P_{dn}(\ell)\nonumber \\&~~~~~~~~~+\sum_{q=1}^QC_{I,q}(a_q(\ell) - d_q(\ell))  ]\rbrace 
\end{align}
$$\mbox{s.t.}~~~ \forall \ell \leq \ell_0+T, q\in \{1,\ldots,Q\}:~~ d_{q}(\ell-1) \leq d_{q}(\ell) \leq a_{q}(\ell)$$
$$~~~d_{q}(\ell) \in  \mathbb{N},~~~ d_{q}(\ell+T-u_q) = a_{q}(\ell+T-u_q),$$
$$\forall \ell \leq \ell_0+T:~~P_{up}(\ell) \geq 0,~~~P_{dn}(\ell) \geq 0,~~~P_{dn}(\ell)P_{up}(\ell) = 0,$$
$$P_{up}(\ell) - P_{dn}(\ell) + P(\ell) = \Re\{L^{S}(\ell,D)\},$$
where $\Re[ L^{S}(\ell,D)]$ represents the active (real) part of the D-load for $\ell>\ell_0$.

After solving the above optimization problem at time $\ell_0$, the scheduler informs the D-loads about its decisions on the optimum departure process $d_q(\ell_0)$ for each queue, and discards the decisions made for the future (the dummy variables). Afterwards, it needs to solve the optimization at time $\ell_0+1$ again. Since we assume that scheduled tasks are not interruptible, the load due to D-loads that were scheduled to start their job at time $\ell_0$ is added to the non-controllable part of the load at and after time $\ell_0+1$, i.e.,for $\ell \geq \ell_0+1$:
$$L^N(\ell) \rightarrow L^N(\ell) + \sum_{q=1}^{Q} [d_q(\ell_0)-d_{q}(\ell_0-1)]g_{q}(\ell-\ell_0).$$

 \begin{remark} The finite horizon assumption is made since ${P}(\ell)$ is non-stationary and our knowledge about it cannot be considered perfect in an infinite horizon. Also, forecast errors for the arrival pattern of customers increases with time.\end{remark}

 \begin{remark} In the following, we will ignore the reactive power load due to the scheduled D-loads and assume that the pulses $g_q(\ell)$ will represent active power requests only (i.e., $\Re[ L^{S}(\ell)] = L^{S}(\ell)$. But, if desired, one can easily incorporate reactive power requirements in the above optimization problem as a constraint:
$$
\Im [L^{S}(\ell)] = \sum_{q=1}^{Q} \sum_{i=\ell_0}^{\ell} (d_q(i)-d_q(i-1))\Im[g_q(\ell-i)]  \leq \epsilon .
$$\end{remark} For brevity, from this point on we omit the dependency on $D$ of both $L^{S}(\ell;D)$ and $\mbox{DCI}(\ell_0,D) $. Next, we relax the problem \eqref{mainopt2}, considering two scenarios for the arrival process. 

\vspace{-0.3cm}
\subsection{Scenario I}\label{scen3}
In this scenario, we look at the case where the arrival process is deterministic but the rate $\lambda_q(t)$ is not constant in time. This problem can be formulated as \eqref{mainopt2} without the expected value operator.

\begin{lemma}If we ignore the case where $C_{dn}(\ell) < 0$, the constraint $P_{dn}(\ell)P_{up}(\ell) = 0$ can be eliminated from \eqref{mainopt2}.\end{lemma}
\begin{proof} (by contradiction). Assume that for epoch $\ell$, the optimum solution to \eqref{mainopt2}  is such that $P^*_{dn}(\ell)P^*_{up}(\ell) > 0$. One can replace these values by setting the smaller number to 0 and the larger one to $|P^*_{dn}(\ell) - P^*_{up}(\ell)|$. The constraints will all still hold while the cost will be reduced. Thus, the optimization problem without the constraint $ P^*_{dn}(\ell)P^*_{up}(\ell) = 0$ is equivalent to  \eqref{mainopt2}.
\end{proof}
%\begin{align} 
%\min_{D}C_{\ell_0}(\bar{x}(\ell_0)) = \min_{D} &\sum_{\ell=\ell_0}^{\ell_0+T} C_{dv}(\ell)|P(\ell) - L^{S}(\ell)| \nonumber\\&~+ \mbox{DCI}(\ell_0,D)\nonumber\\
%\mbox{s.t.}~~~ \forall \ell \geq \ell_0,~~~~ &d_q(\ell-1) \leq d_q(\ell) \leq a_q(\ell)\label{opt6}\\
%&d_{q}(\ell) \in  \mathbb{N},~~d_{q}(\ell_0 + T) = a_{q}(\ell_0 + T).\nonumber
%\end{align}
%\end{lemma}
%\begin{proof}
%The upward and downward deviation costs $C_{up}(\ell)$ and $C_{dn}(\ell)$ are approximated with a single deviation cost $C_{dv}(\ell)$. This step will allow next to relax the problem into a linear program. 
%\end{proof}

To continue, we define the following matrices,
\begin{eqnarray}\label{defv}
D &=& [\bar{d}(\ell_0),\ldots , \bar{d}(\ell_0+T)]\nonumber\\ A&=&[\bar{a}(\ell_0),\ldots , \bar{a}(\ell_0+T)]\nonumber\\P&=& [P(\ell_0),\ldots ,P(\ell_0+T)]\nonumber\\
P_{up}&=& [P_{up}(\ell_0),\ldots ,P_{up}(\ell_0+T)]\nonumber\\
P_{dn}&=& [P_{dn}(\ell_0),\ldots ,P_{dn}(\ell_0+T)]\nonumber\\
C_{up} &=& [C_{dv}(\ell_0),\ldots,C_{up}(\ell_0+T)]\nonumber\\
C_{dn} &=& [C_{dn}(\ell_0),\ldots,C_{dn}(\ell_0+T)]\nonumber\\
C_{T} &=&  [C_{I,1},\ldots,C_{I,Q}].
\end{eqnarray}
\begin{lemma} \eqref{mainopt2} can be cast in the following constrained mixed-integer  linear programming problem:
\begin{equation} \label{opt7}
\min_{D,P_{up},P_{dn}} C_{up}P_{up}^T + C_{dn}P_{dn}^T - (C_T\otimes\mathbf{1}_{1\times L})\mbox{vec}(D^T)
\end{equation}
$$\mbox{s.t.}~~~  0 \preceq \mbox{vec}(D^T) \preceq \mbox{vec}(A^T),~ (I\otimes J^T)\mbox{vec}(D^T) \succeq 0,$$
$$ P_{up}\succeq 0, ~~ P_{dn}\succeq 0,~~\Gamma \mbox{vec}(D^T) + P^T + P_{up}^T - P_{dn}^T = 0, $$
$$ D \in  \mathbb{N}^{Q\times (T+1)}$$
where $\mbox{vec}$ denotes the vectorization operation. The matrix $\Gamma = [\Gamma_1~,\Gamma_2,~\ldots ~,\Gamma_Q](I\otimes J^T)$, where $J$ is the first difference operator,
% i.e., 
%\begin{equation} J = \left[ \begin{array}{cccccc}
%1 & -1 & 0 &\ldots  & 0 & 0\\
%0 & 1 & -1 & 0 & \ldots  & 0\\
%\vdots & \ddots & \ddots & \ddots & \ddots & \vdots \\
%0 & \ldots & 0 & 1 & -1 & 0\\
%0 & \ldots & 0 & 0 & 1 & -1\\
%0 & \ldots & 0 & 0 & 0 & 1\\
%\end{array} \right] 
%\end{equation}
and the matrices $\Gamma_q$ are the Toeplitz matrices associated with each $g_q(\ell) = g(\ell;C_q)$, with first row equal to
$[g_q(1), 0, \ldots , 0 ]$ and first column equal to $[g_q(1),g_q(2),\ldots,g_q(T)]$.
%\begin{equation} \Gamma_q = \left[ \begin{array}{ccccccc}
%g_q(1) & 0 & \ldots  & 0 & \ldots& 0 \\
%g_q(2) & g_q(1) &  &  &  &  \\
%g_q(3) & g_q(2) & \ddots & & &\vdots   \\
%\vdots & \vdots & \ddots & \ddots& &0  \\
%g_q(T -2) & g_q(T -3) & \ddots &  g_q(1) & 0 & \vdots \\
%g_q(T -1) & g_q(T -2) &  & g_q(2)  &  g_q(1) &0\\
%g_q(T) & g_q(T -1) &  & g_q(3)   &  g_q(2)&  g_q(1)\\
%\end{array} \right].
%\end{equation}
\end{lemma}
\begin{proof} We showed that the load due to D-loads is the sum of the following $Q$ functions:
\begin{equation}\left. \begin{array}{ll}
\forall q: ~L^{S}_q(\ell)= \sum_{i=\ell_0}^{\ell_0+T}  (D(q,i) - D(q,i-1))g_q(\ell-i)
\end{array} \right.
\end{equation}
which is basically the convolution of first difference of the $q^{th}$ row of $D$ with the vector $[g_q(1) ,\ldots, g_q(T)]$. Thus, using $Q$ Toeplitz matrices, each associated with one of the $g_q(\ell)$'s, and considering that $\mbox{vec}((DJ)^T)=(I\otimes J^T) \mbox{vec}(D^T)$, we can rewrite the load due to D-loads as the following vector
\begin{eqnarray}
L^{S} &=& [\Gamma_1~,\Gamma_2,~\ldots ~,\Gamma_Q](I\otimes J^T) \mbox{vec}(D^T). \nonumber
\end{eqnarray}
On the other hand, the DCI in \eqref{delay5} can be rewritten as
\begin{equation}\left. \begin{array}{ll}
\Vert C_T (A - D) \Vert_1 = (C_T\otimes\mathbf{1}_{1\times L}) (\mbox{vec}(A^T)-\mbox{vec}(D^T)). \end{array} \right.
\end{equation}
With the first term independent of the decision variables and thus, not contributing to the objective function minimization, we can add only the second term to the cost to represent the delay cost. The rest of the constraints in \eqref{opt7} are just transformations of the constraints in \eqref{mainopt2} to vector form. 
\end{proof}

Since large-scale integer programs like \eqref{opt7} are complex to solve, we may choose to relax the problem by omitting the integrality constraints for the decision variables. The linear program obtained gives a lower bound on the optimal value of the integer linear programming problem (the feasible
region of the integer program is a subset of its corresponding LP relaxation problem) After solving the relaxed problem, we can round the fractional values to obtain integral solutions. A possible modification is to try relaxations based branch and bound algorithms and recurse into several relaxed subproblems until a good solution is found. Due to the large size of our problem, we resort to the simpler rounding solution.

%\begin{lemma} The solution of \eqref{opt7} can be approximated by the solution of the following constrained least square problem,
%\begin{equation} \label{opt8}\left. \begin{array}{ll}
%\min_{D\in \mathcal{D}} \Vert
%y - 
%F \mbox{vec}(D')\Vert_2 
%\\
%\\
%\mbox{s.t.}~~~  \mbox{vec}(D') \preceq \mbox{vec}(A')\\
%~~~~~~ (I\otimes J')\mbox{vec}(D') \succeq 0\\
%~~~~~~ \mbox{vec}(D') \succeq 0,
%\end{array} \right.
%\end{equation}
%where 
%\begin{equation}\left. \begin{array}{ll}
%F = \left[ \begin{array}{ll}
%C_{DV}T \\
%(C_T \otimes I)
%\end{array} \right]~~
%and ~~
%y = \left[\begin{array}{ll}
%C_{DV}P\\
%(C_T \otimes I)\mbox{vec}(A')
%\end{array} \right]. 
%\end{array} \right. \nonumber
%\end{equation}
%\end{lemma}
%\begin{proof}
%If we approximate the $\mathcal{L}_1$ norm with the $\mathcal{L}_2$ norm and relax the constraints on the decisions having integer values, we get the above problem formulation.
%\end{proof}
%This is thus an LSI problem (least square with inequality constraints). Several methods have been proposed to solve this problem numerically \cite{bjorck,zhu}. Please see \cite{Hanson} for a detailed description.

\vspace{-0.4cm}
\subsection{Scenario II}
In this scenario, customers arrive according to a random arrival process with a non-homogeneous rate $\lambda(t)$, e.g. a Poisson process. Since we no longer have deterministic information about future arrivals and only the past is fully observable for us, we need to solve a stochastic shortest path problem:
\begin{eqnarray} \label{opt12}
C^{*}_{\ell_0}(\bar{x}(\ell_0)) = %\min_{D\in \mathcal{D}} C_{\ell_0}(\bar{x}(\ell_0) \\&=&
\min_{\bar{d}} \sum_{\ell=\ell_0}^{\infty}  \operatorname{E}_{\bar{a}} \left\{c(\bar{x}(\ell),\bar{d}(\ell),\bar{a}(\ell))\right\},%\nonumber
\end{eqnarray}
where the vector $\bar{x}(\ell)$, representing the state (expanded to include previous decisions), and the function $c(.)$ representing the cost incurred at time $\ell$ are defined as
$$\bar{x}(\ell) = [P(\ell),\bar{d}(\ell-\min\lbrace \ell-\ell_0, S\rbrace),\ldots,\bar{d}(\ell-1)]^T,$$
\begin{eqnarray} 
c(\bar{x}(\ell),\bar{d}(\ell),\bar{a}(\ell))&=& C_{up}(\ell)P_{up}(\ell)+ C_{dn}(\ell)P_{dn}(\ell)\nonumber\\&+& \sum_{q=1}^QC_{I,q}(a_q(\ell) - d_q(\ell)).\nonumber
\end{eqnarray} 
A state variable is the minimal history that is
necessary and sufficient to compute the decision function \cite{castlelab}. For our case, this includes all the decisions made in the past that can affect the current load of the system and the ZIC power profile in \eqref{eq.zic}. If we assume that the appliances can be on for a maximum duration of $S\triangle$ units of time, the last $S$ decisions should be included in the state variable.
%where $\bar{x}(k)$ and the cost function $c(.)$ have been introduced in section \ref{scen3}. As it is easily seen, the state space for this problem is countably infinite.

As our knowledge about the intermittent resources and the arrival trends of D-loads will degrade significantly with time, solving \eqref{opt12} is not practical. Based on this fact, we divide our look-ahead horizon into the categories below:
\\[0.1cm]
1) Full statistical knowledge interval: in this interval, we have full statistical information about the arrival process and the amount of available resource and thus, we exploit this knowledge in our decision strategy.
\\[0.1cm]
2) Imperfect information interval: during this interval, we have some information about the arrival and generation processes but our knowledge is not as accurate as the first interval and thus, we will only use the expected value of predictions.
\\[0.1cm]
3) No useful knowledge interval: after a certain amount of time, our knowledge about the arrival process and the resource will degrade so significantly that we would rather not incorporate them in our decision.

Hence, we redefine our optimized scheduling problem as:
\begin{equation} \label{opt10}
\min_{D\in \mathcal{D}} \operatorname{E}_{\bar{x}}\left\{ \sum_{\ell=\ell_0}^{\ell_0+T-1} c(\bar{x}(\ell),\bar{d}(\ell),\bar{a}(\ell)) +\hat{C}_{\ell_0+T}(\bar{x}(\ell_0+T))\right\}
\end{equation}
where $\hat{C}_{\ell_0+T}(x(\ell_0+T))$ is the approximation of the true cost-to-go function $C^{*}_{\ell_0+T}(x(\ell_0+T))$ after $T$ stages (the length of the full statistical knowledge interval). This approximation is found using a base suboptimal policy that we solve for assuming that the uncertain quantities in the imperfect information interval are fixed at their typical values. Thus, this is an $T$-step lookahead policy, with the optimal cost-to-go approximated by the cost-to-go of a certainty equivalent controller. Note that we have already discussed the possible strategies to solve this problem in scenario III. Also, we can reduce the number of possible states by using state aggregation strategies in the complete statistical knowledge interval in order to construct a more simple and tractable problem. This will make the computation much less demanding.

\begin{remark} We solved the scheduling problem by ignoring flow constraints of the distribution network. Since we do not take individual D-load delays into account and we only care about the total delay experienced by the customers, these flow constraints can be incorporated in the design with a simple modification using the following suboptimal policy: If the BS knows that it does not have distribution capacity to let any more appliances start functioning, it sends a full capacity message to the CEMS, which can be incorporated into  \eqref{alpha2}.\end{remark}

\vspace{-0.2cm}
\section{Numerical Results}\label{numerical}
In this section, we compare the effectiveness of the introduced scheduling policies exercised at a cell that controls the charging of a fleet of around 15,000 vehicles to the following two scenarios: the case where no control mechanism is provided over the load and the case where the control is applied in a more distributed fashion with 500 local schedulers spread across the area where the program is implemented. We assumed that the following information are available to the scheduler: 1) perfect predictions of the mean arrival rate of vehicles $\lambda(t)$; 
2) causal observability of the arrival process $a(t)$;
3)  predictions of the ZIC power profile \eqref{eq.zic} in the lookahead horizon;
4) predictions of the upward and downward clearing prices in the lookahead horizon; and
5) accurate estimates of the load injections $g_q(t)$.
To avoid computational complexities and to make the resulting demand profiles more meaningful to the reader, we assumed that the charging duration of the vehicles has a maximum of 8 hours and is quantized in 15 minute intervals, which is equal to the frequency with which the scheduler solves the optimization problem. The appliances arrive in the system following a Poisson arrival process with a mean arrival rate $\lambda = 12$ per hour for each of the 32 queues, assumed to be constant, in order for the results to be interpretable. The optimizers solve the scheduling problem with a certainty equivalent controller that uses a relaxed linear programming approach to determine the best scheduling strategy. After each step, the scheduled D-loads are added to the uncontrollable load term for future epochs.  The upward and downward balancing prices were chosen to be equal and the charging pulses were assumed to be square pulses scaled to match the duration of charging. 

Fig. \ref{sim} compares the demand in the following two scenarios: coordinated charging of all vehicles by one scheduler unit and uncontrolled charging.  The look-ahead horizon in these simulations is assumed to be 8 hours and to be fair, no appliance is allowed to be delayed beyond 8 hours after it's arrival, which in reality can differ for different queues. Thus, the number of appliances receiving service is equal in the two scenarios. The results show that the average delay experienced by customers if everybody participates in the DDLS  program is less than one unit of time and due to the considerable decrease in the deviation costs, there is a relative 41 percent reduction in operational costs compared to the uncontrolled case. 
%One can note that even if all of these arrivals (an average of 384 arrivals every hour) are covered by one base station, the number of bits that need to be communicated to support this load control strategy are bound by 5 and 10 bits/sec respectively for the HEMS-MAC uplink and the downlink, which are very modest.
\begin{figure}
\centering
\includegraphics[scale=0.37]{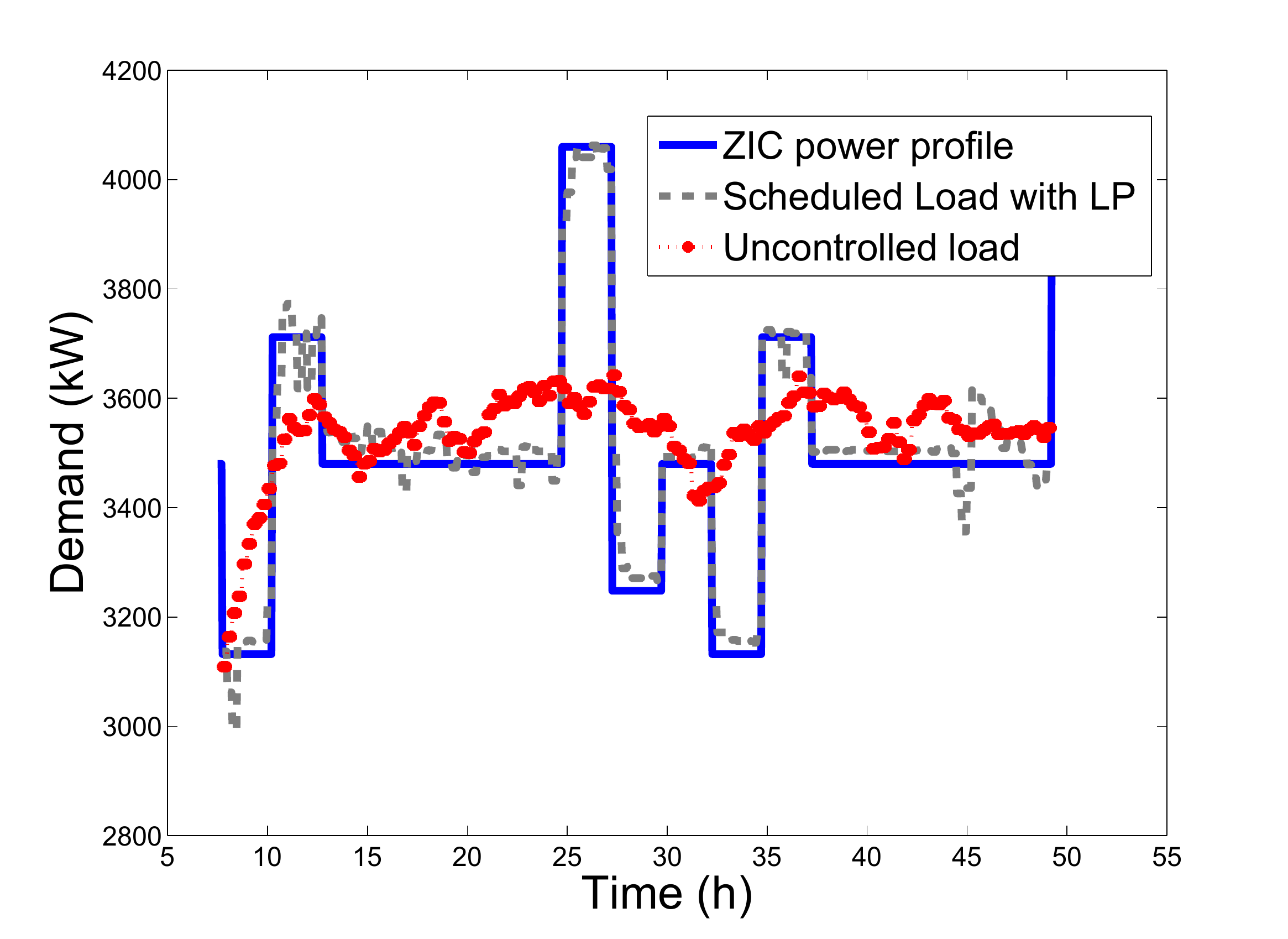}
\caption{\small DDLS allows to follow a certain generation profile}
\label{sim}
\end{figure}

Fig. \ref{sim2} compares the case where instead of a single scheduler unit, several schedulers are distributed across the area where the DDLS program is carried out (20 schedulers in our case, each scheduling around 20 EVs per hour). To showcase the more subtle differences in this comparison, we have scaled down our simulation horizon to 16 hours and decreased the arrival rate by half. The savings that resulted from the distributed schedulers were still considerable (21 percent).
 
\begin{figure}
\centering
\includegraphics[scale=0.34]{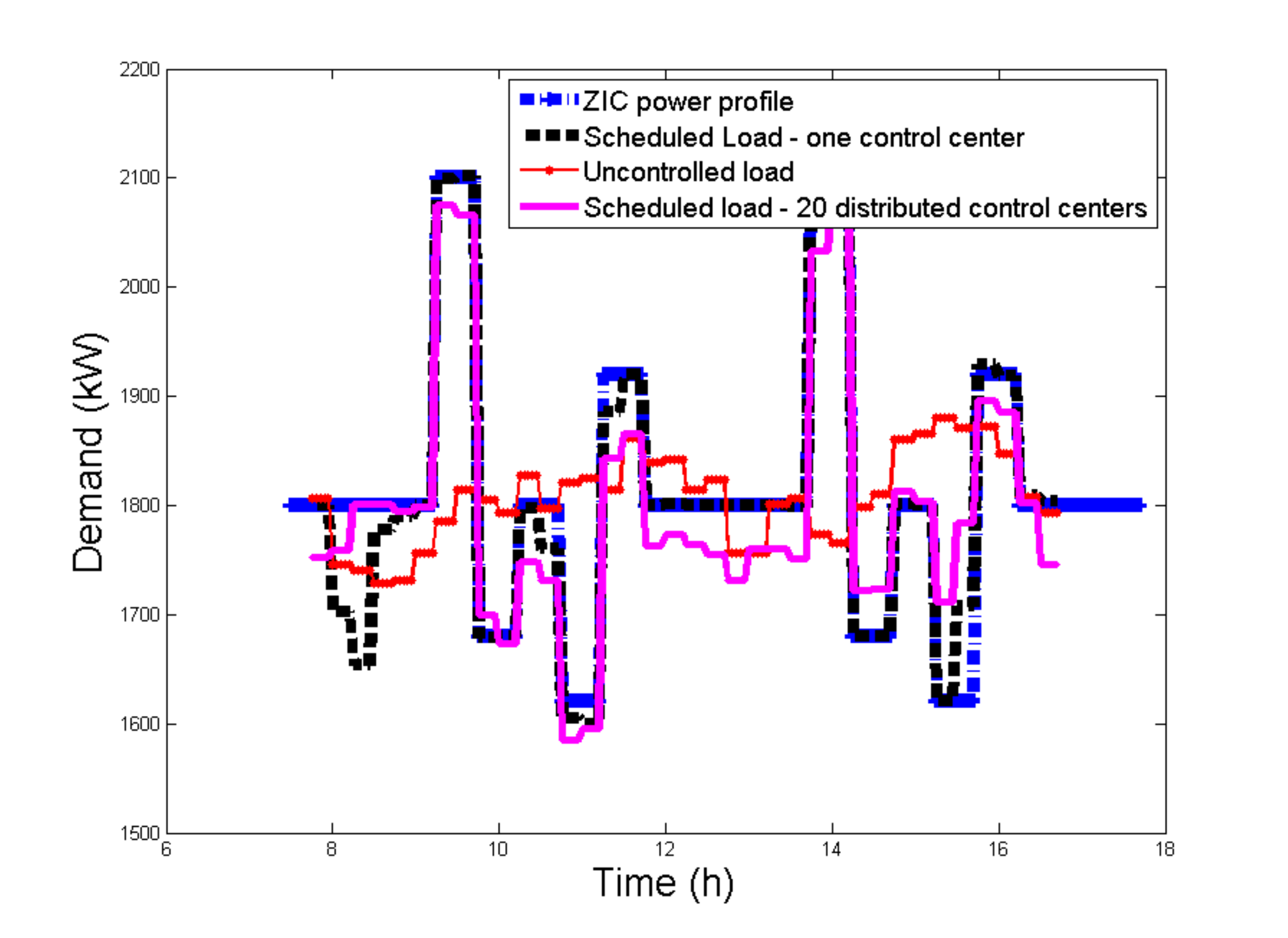}
\caption{\small Higher level of aggregation leads to greater savings}
\label{sim2}
\end{figure}
In Fig. \ref{price}, we compare the load directly controlled by the DDLS scheduler to the case where a single price signal is broadcast to every vehicle, using which they locally minimize their cost for charging. Due to a lack of work that examines how to optimally design the pricing signal and, since the comparison is based on how well the controlled load can follow the day-ahead bid, the price signals were chosen to be a function of excess or shortfall of generation from the uncontrolled load profile (which is close to flat). Individual customers are provided with the price for the whole day. Several price functions were tested but they all led to similar results. Since customers act selfishly to minimize their own costs, this technique results in very high rebound peaks, occurring at times when the price is lower even by a small amount (due to small increases in the ZIC power profile in \eqref{eq.zic}). The conclusion that can be drawn from this result is that implementing pricing techniques, irrespective of how good the individual HEMS work, needs much more research. Also, sending a single price signal to all the customers can endanger the stability of the grid.
\begin{figure}
\centering
\includegraphics[scale=0.36]{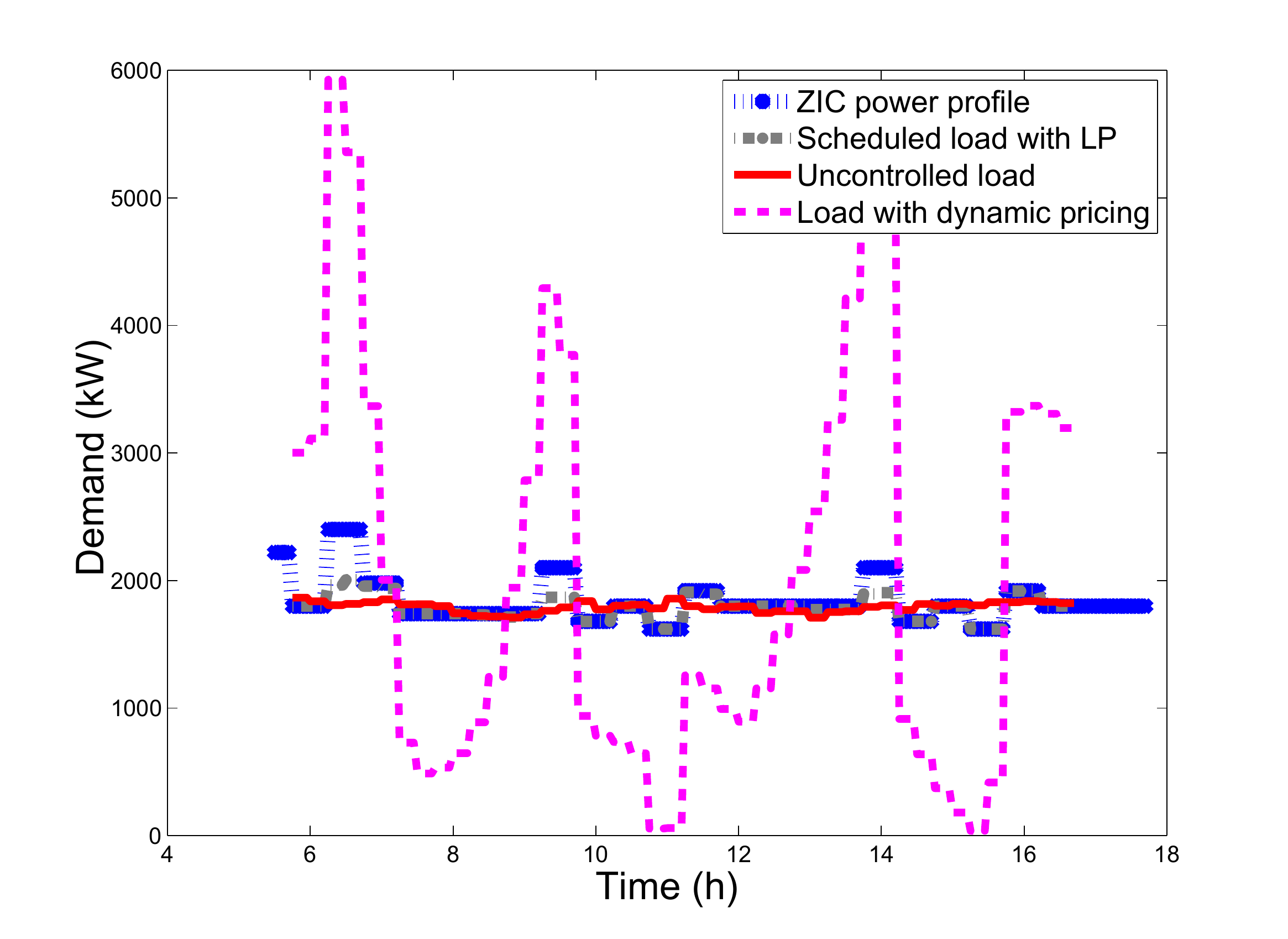}
\caption{\small A single pricing signal for everybody results in rebound peaks}
\label{price}
\end{figure}

\vspace{-0.3cm}
\section{Conclusions and future work}
\vspace{-0.1cm}
In this paper, we introduced a novel DSM model for direct load scheduling. We provided a communication framework to defer the time at which loads turn on, after informing the control center about their arrival in the grid and service request. We described a cellular architecture that we consider suitable to integrate Digital Direct Load Scheduling programs into the grid. Finally, we formulated and solved the problem of managing the energy demand from loads that can be deferred during the real-time operation of the cell. Simulations showed that such a program can alleviate the problem of matching demand and generation considerably with minimal inconvenience to the customers and parsimonious communication needs. In our future work, we will study architectures to incorporate these large reservoirs of storable demand in the electricity market. We will also look more closely at the statistical information that is required to solve the scheduling problems, e.g. the arrival patterns of D-loads, their requests and the availability of renewables. Another major area for future work would be to look at how to integrate the degrees of freedom offered by thermostatically controlled appliances in a neighborhood DSM program (maybe with smaller coverage areas than what was considered in this paper).
\renewcommand{\baselinestretch}{1}

\bibliographystyle{IEEEtran}
\small
\bibliography{New2}

\begin{biography}[{\includegraphics[width=1in,height=1.25in,clip,keepaspectratio]{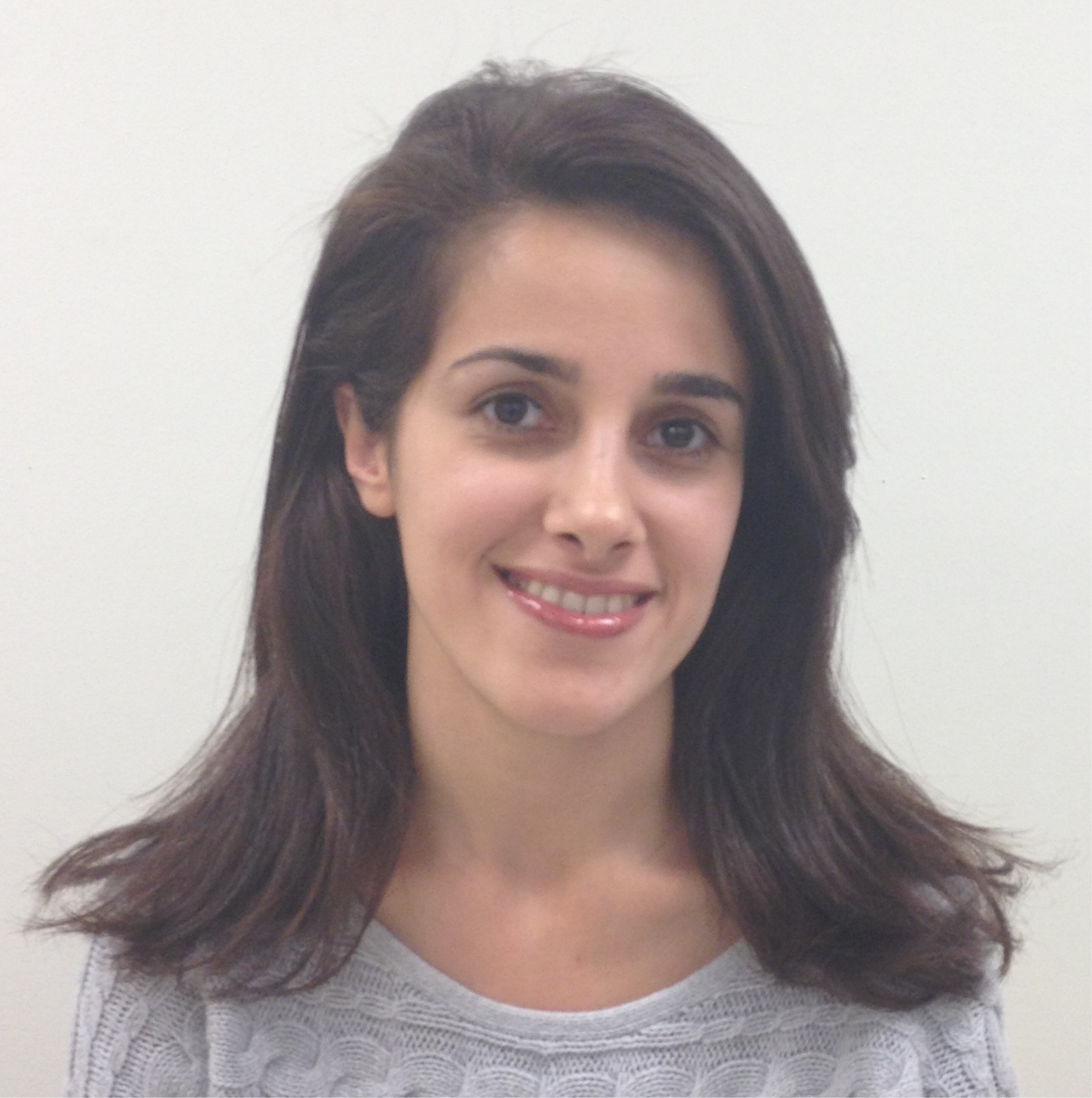}}]{Mahnoosh Alizadeh}(S' 08) has been a PhD student in University of California, Davis since 2009. She received her B.S. degree in Electrical Engineering from Sharif University of Technology. Her research interests mainly lie in the area of Smart Grid, in particular demand side management and demand response.
\end{biography}

\begin{biography}[{\includegraphics[width=1in,height=1.25in,clip,keepaspectratio]{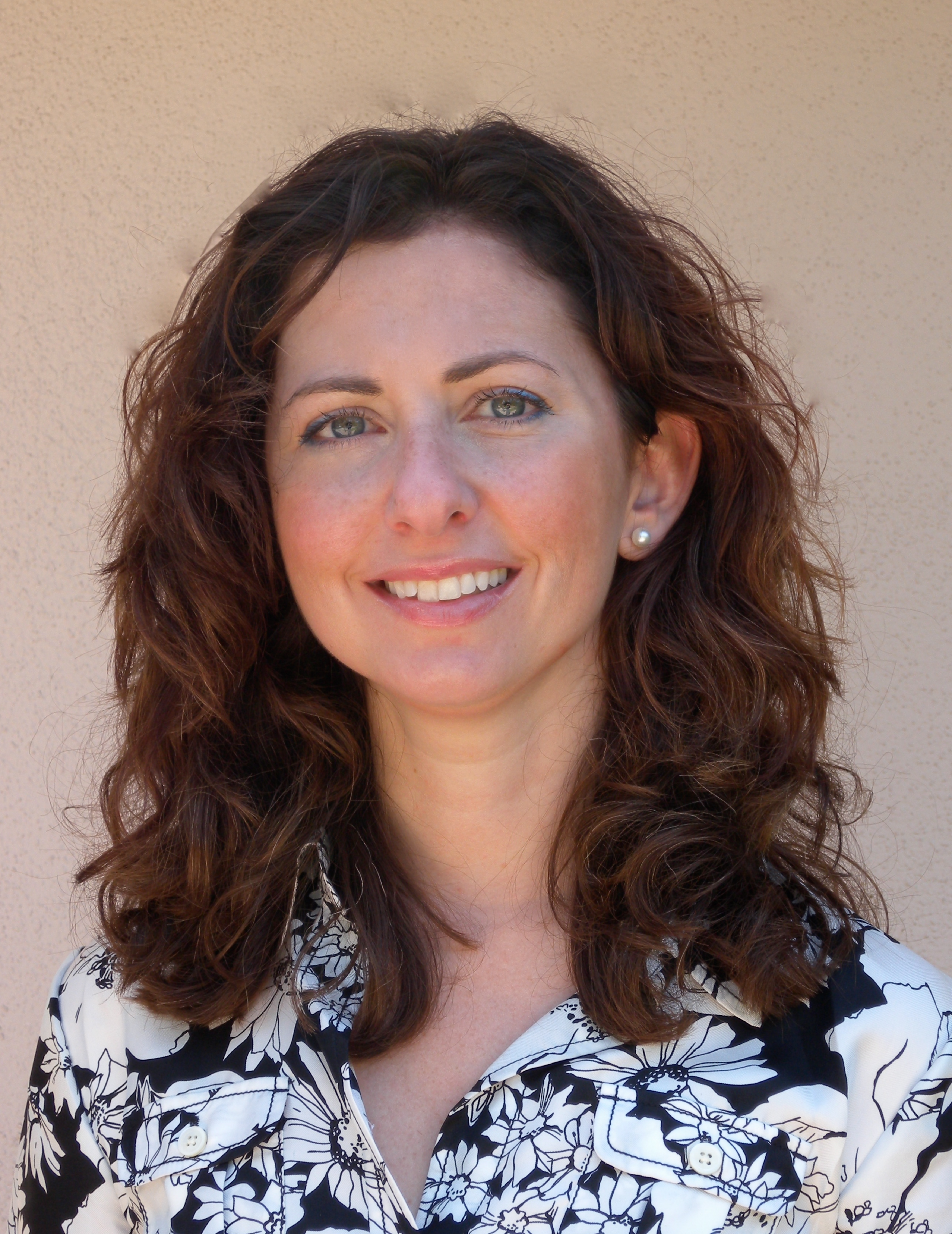}}]{Anna Scaglione} 
(S'97-M'99-SM'09-F'11) received the
Laurea (M.Sc. degree) and the Ph.D. degree from the
University of Rome La Sapienza,� Italy, in 1995 and
1999, respectively.
She was previously at Cornell University from
2001 to 2008, where she became Associate Professor
in 2006; before joining Cornell, she was Assistant
Professor from 2000 to 2001, at the University
of New Mexico. Since July 2008, she has been
Associate Professor in Electrical and Computer
Engineering at the University of California, Davis.
Her expertise is in the broad area of signal processing for communication
systems and networks. Her current research focuses on cooperative wireless
networks and sensors� systems for monitoring and control applications
\end{biography}

\begin{biography}[{\includegraphics[width=1in,height=1.25in,clip,keepaspectratio]{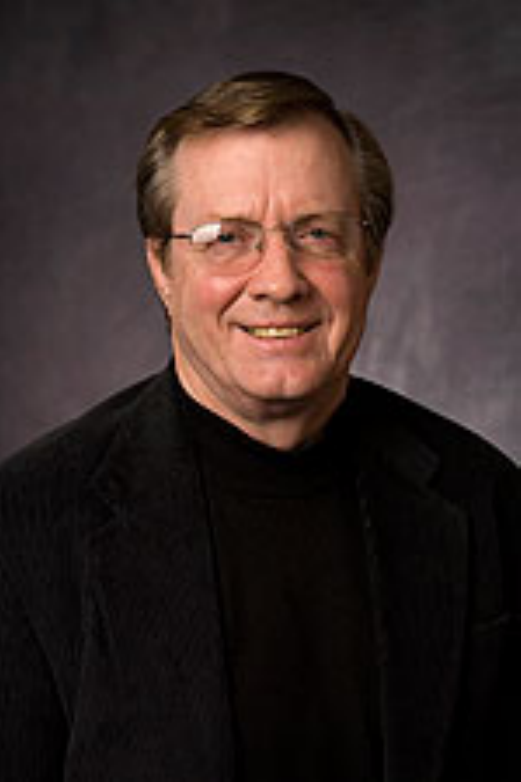}}]{Robert J. Thomas} 
(M'73-SM'83-F'93-LF'08) received the Ph.D. degree in electrical engineering from
Wayne State University, Detroit, MI, in 1973.
He is currently Professor Emeritus of Electrical
and Computer Engineering at Cornell University,
Ithaca, NY. His technical background is broadly
in the areas of systems analysis and control of
large-scale electric power systems. He has published
in the areas of transient control and voltage collapse
problems as well as technical, economic, and institutional impacts of restructuring.
Prof. Thomas is a member of Tau Beta Pi, Eta Kappa Nu, Sigma Xi, and
ASEE. He has received five teaching awards and the IEEE Centennial and Millennium medals. He has been a member of the IEEE-USA Energy Policy Committee since 1991 and was the committee's Chair from 1997-1998. He is the
founding Director of the 13-university-member National Science Foundation
Industry/University Cooperative Research Center, PSERC and he currently serves
as one of 30 inaugural members of the U.S. Department of Energy Secretary's
Electricity Advisory Committee (EAC).
\end{biography}

\end{document}